\documentclass{amsart}
\usepackage{amssymb,amscd,amsthm,indentfirst,amsfonts,amsmath,geometry,setspace,hyperref,graphicx}
\usepackage[initials]{amsrefs}

\newtheorem{theorem}{Theorem}[section]
\newtheorem{proposition}[theorem]{Proposition}

\newtheorem{remark}[theorem]{Remark}

\newtheorem{lemma}[theorem]{Lemma}

\numberwithin{equation}{section}

\begin{document}

\title{Maximal lineability of the set of continuous surjections}
\author{Nacib Gurgel Albuquerque}
\address{Departamento de Matem\'{a}tica, \newline \indent
Universidade Federal da Para\'{i}ba, \newline \indent
58.051-900 - Jo\~{a}o Pessoa, Brazil.}
\email{ngalbqrq@gmail.com}

\thanks{The author is supported by Capes}



\begin{abstract}
Let $m,n$ be positive integers. In this short note we prove that the set of all continuous and surjective functions from $\mathbb{R}^{m}$ to $\mathbb{R}^{n}$ contains (excluding the $0$ function) a $\mathfrak{c}$-dimensional vector space. This result is optimal in terms of dimension.
\end{abstract}

\maketitle

\section{Preliminaries}

Lately the study of the linear structure of certain subsets of surjective functions in $\mathbb{R}^\mathbb{R}$ (such as everywhere surjective functions, perfectly everywhere surjective functions, or Jones functions) has attracted the attention of several authors working on Real Analysis and Set Theory (see, e.g. \cites{ags2005,aronseoane2007,gamezmunozsanchezseoane2010,gamezmunozseoane2010,bps2013}). The previously mentioned functions are, indeed, very ``pathological'': for instance an everywhere surjective function $f$ in $\mathbb{R}^\mathbb{R}$ verifies that $f(I) = \mathbb{R}$ for every interval $I \subset \mathbb{R}$ and the other classes (perfectly everywhere surjective functions and Jones functions) are particular cases of everywhere surjective functions and, thus, with even ``worse'' behavior. It has been shown \cite{gamez2011} that there exists a $2^{\mathfrak{c}}$-dimensional vector space every non-zero element of which is a Jones function and, thus, everywhere surjective (here, $\mathfrak{c}$ stands for the cardinality of $\mathbb{R}$). Of course, this previous result is optimal in terms of dimension since dim($\mathbb{R}^\mathbb{R}$)$=$ $2^{\mathfrak{c}}$. However, all the previous classes are nowhere continuous, thus, it is natural to ask about the set of continuous surjections. The aim of this short note is to prove, in a more general framework that of $\mathbb{R}^\mathbb{R}$, that (for every $m,n \in \mathbb{N}$) the set of continuous surjections from $\mathbb{R}^{m}$ onto $\mathbb{R}^{n}$ is $\mathfrak{c}$-lineable \cite{ags2005} (that is, it contains a $\mathfrak{c}$-dimensional vector space every non-zero element of which is a continuous surjective function from $\mathbb{R}^{m}$ onto $\mathbb{R}^{n}$). Since $\dim \mathcal{C}\left(\mathbb{R}^{m},\mathbb{R}^{n}\right)  = \mathfrak{c}$ we have that this result would be the best possible in terms of dimension, that is, the set of continuous surjections from $\mathbb{R}^{m}$ onto $\mathbb{R}^{n}$ is maximal lineable \cite{bernal2010}.

While there are many trivial examples of surjective continuous functions in $\mathbb{R}^\mathbb{R}$, coming up with a concrete example of a continuous surjective function from $\mathbb{R}$ onto $\mathbb{R}^2$ is a totally different story. The existence of a continuous surjection from $\mathbb{R}$ onto $\mathbb{R}^{2}$ (a \emph{Peano type} function) can be found in \cite[p. 42]{kharazishvili} or \cite[p. 274]{munkres}. Both references use the existence of a continuous surjection from $\left[0,1\right]$ onto $\left[  0,1\right]^{2}$ (a \emph{Peano curve} in $\left[  0,1\right]^{2}$ or a \emph{space filling curve}). The existence of this curve is proved, for instance, in \cite{kharazishvili} invoking a result due to A. D. Alexandrov: there is a continuous surjection from the Cantor space $\mathcal{K}$ onto any arbitrary nonempty compact metric space (see \cite[p. 40]{kharazishvili}); in \cite[section 44]{munkres} the construction of the Peano curve is done geometrically, and is a consequence of the completeness of the space $\mathcal{C}(X,M)$ of all continuous functions from a topological space $X$ to a complete metric space $M$, considering $\mathcal{C}(X,M)$ with the uniform metric.

\section{The lineability of the set of continuous surjections from $\mathbb{R}^{m}$ to $\mathbb{R}^{n}$}

Let $m$ and $n$ be positive integers. Throughout this note we shall denote
$$
\mathcal{S}_{m,n} = \left\{  f\,:\,\mathbb{R}^{m}\longrightarrow\mathbb{R}^{n} \;;\;f\mbox{ is continuous and surjective}\right\}.
$$

The following result shows that $\mathcal{S}_{m,n} \neq \varnothing$, and uses the fact that $\mathcal{S}_{1,2} \neq \varnothing$ (\cite[p. 42]{kharazishvili}).

\begin{proposition}
Let $m, n \in \mathbb{N}$. There exists a continuous surjection $f: \mathbb{R}^{m} \rightarrow \mathbb{R}^{n}$.
\end{proposition}

\begin{proof}
Let us take $f \in \mathcal{S}_{1,2}$. If $f_i := \pi_i \circ f, \; i=1,2$ denotes the $i$-coordinates functions of $f$ ($f=(f_1,f_2)$), then the map $id_{\mathbb{R}}\times f\,:\,\mathbb{R}^{2}\longrightarrow\mathbb{R}^{3}$
defined by $id_{\mathbb{R}}\times f(t,s) := (t,f_1(s),f_2(s))$ is a continuous surjection. Thus, $(id_{\mathbb{R}}\times f)\circ f$
is in $\mathcal{S}_{1,3}$. Proceeding in an induction manner, we can assure the existence of a function $g$ belonging to $\mathcal{S}_{1,n}$ for every $n\in \mathbb{N}$. Hence, defining $F\,:\, \mathbb{R}^{m} \longrightarrow \mathbb{R}^{n}$ by $ F := g \circ \pi_1$, i.e.,
$$
F(x)=F(x_1,\ldots,x_m)=g(x_1) \;,\; \mbox{ for all } x=(x_1,\dots,x_m) \in \mathbb{R}^m
$$
($\pi_1 : \mathbb{R}^m \longrightarrow \mathbb{R}$ denotes the canonical projection over the first coordinate), we conclude that $F \in \mathcal{S}_{m,n}$ ($F$ is composition of continuous surjective functions).
\end{proof}

Attempting maximal lineability of $\mathcal{S}_{m,n}$ (that is, $\mathfrak{c}$-lineability) we make use of the following remark (inspired in a result from \cite{ags2005}), which indicates a method to obtain our main result.

\begin{remark} \label{remark2.2}
Given a continuous surjection $f\,:\,\mathbb{R}^{m}\longrightarrow\mathbb{R}^{n}$, suppose we have $\mathcal{X}\subset\mathcal{C}\left(  \mathbb{R}^{n};\mathbb{R}^{n}\right)$ a subset of $\mathfrak{c}$-many linearly independent functions such that every nonzero element of $\mbox{span}(\mathcal{X})$ is a continuous surjection. Then, we have that
$$
\mathcal{Y} := \{F\circ f\}_{F\in\mathcal{X}}\subset\mathcal{C}\left( \mathbb{R}^{m};\mathbb{R}^{n}\right)
$$
has cardinality $\mathfrak{c}$, is linearly independent and is formed just by continuous surjections. Moreover,
$$
\mbox{span}(\mathcal{Y})\subset\mathcal{S}_{m,n}\cup\{0\},
$$
obtaining the $\mathfrak{c}$-lineability of $\mathcal{S}_{m,n}$.
\end{remark}

In order to continue we shall need two lemmas and some notation. First, let us consider (for $r>0$) the homeomorphism $\phi_{r} : \mathbb{R}\rightarrow \mathbb{R}$ given by
$$
\phi_{r}(t) := e^{rt} - e^{-rt}.
$$

\begin{lemma}
The subset $\mathfrak{A} := \{\phi_r\}_{r \in \mathbb{R}^+}$ of $\mathbb{R}^\mathbb{R}$ is linearly independent, has cardinality $\mathfrak{c}$, and every nonzero element of $\mbox{span}(\mathfrak{A})$ is continuous and surjective.
\end{lemma}

\begin{proof}
First let us prove that every nonzero element $\phi = \sum_{i=1}^k \alpha_i \cdot \phi_{r_i} \in \mbox{span}(\mathfrak{A})$ is surjective. We may suppose that $r_1 > r_2 > \cdots > r_k \, \mbox{ and } \, \alpha_1 \neq 0$. Writing
$$
\phi(t) = e^{r_1 t} \cdot \left(\alpha_1 + \sum_{i=2}^k \alpha_i \cdot e^{(r_i - r_1)t}\right) - \sum_{i=1}^k \alpha_i \cdot e^{-r_i t},
$$
we conclude that $\displaystyle \lim_{t\rightarrow+\infty} \phi(t) = \mbox{sign}(\alpha_1) \cdot \infty$ and $\displaystyle \lim_{t\rightarrow-\infty} \phi(t) = - \mbox{sign}(\alpha_1) \cdot \infty$. Thus, the continuity of $\phi$ assures its surjection. Now let us see that $\mathfrak{A}$ is linearly independent: suppose that $\psi = \sum_{i=1}^n \lambda_i \cdot \phi_{s_i} =0$. If there is some $\lambda_j \neq 0$, we may suppose that $s_1 > \cdots > s_n$ and $\lambda_1 \neq 0$. Repeating the argument above, we obtain
$$
\lim_{t\rightarrow+\infty} \psi(t) = \mbox{sign}(\lambda_1) \cdot \infty \, \mbox{ and } \, \lim_{t\rightarrow-\infty} \psi(t) = - \mbox{sign}(\lambda_1) \cdot \infty,
$$
which contradicts $\psi = 0$. This proves that $\mathfrak{A}$ is linearly independent. The other assertions are easy to prove.
\end{proof}

For each $r = (r_1,\ldots,r_n) \in (\mathbb{R}^+)^n$, let $\varphi_r$ be the homeomorphism from $\mathbb{R}^n$ to $\mathbb{R}^n$ defined by $\varphi_r = (\phi_{r_1},\ldots,\phi_{r_n})$, \emph{i.e.},
$$
\varphi_r(x) := (\phi_{r_1}(x_1),\ldots,\phi_{r_n}(x_n)) , \; \mbox{ for all } x=(x_1,\ldots,x_n) \in \mathbb{R}^n.
$$

Working on each coordinate, and using the previous lemma, we have the following.

\begin{lemma}\label{lemma2.4}
The set $\mathfrak{B} = \{\varphi_r\}_{r \in (\mathbb{R}^+)^n}$ of $\mathcal{C}(\mathbb{R}^n;\mathbb{R}^n)$ is linearly independent, has cardinality $\mathfrak{c}$, and every nonzero element of $\mbox{span}(\mathfrak{B})$ is continuous and surjective.
\end{lemma}

Now it is time to state and prove our main result.

\begin{theorem}
${\mathcal S}_{m,n}$ is $\mathfrak{c}$-lineable.
\end{theorem}

\begin{proof}
Let $f \in \mathcal{S}_{m,n}$. Using the notation of the previous lemma and the ideas of the Remark \ref{remark2.2}, we now prove that the set $\mathfrak{C} = \{ F \circ f \}_{F \in \mathfrak{B}}$ is so that  $\mbox{span}(\mathfrak{C})$ is the space we are looking for.

The surjectivity of $f$ assures that $G \circ f = 0$ implies $G=0$, for every function $G$ from $\mathbb{R}^n$ to $\mathbb{R}^n$. Thus, if $G_i \in \mathfrak{B} , \; i=1,\ldots,k$ and
$$
0 = \sum_{i=1}^k \alpha_i \cdot G_i \circ f = \left(\sum_{i=1}^k \alpha_i G_i \right) \circ f,
$$
then $\sum_{i=1}^k \alpha_i \cdot G_i =0$; so since $\mathfrak{B}$ is linearly independent, we conclude that $\alpha_i=0 , \; i=1,\ldots,k$ and thus, $\mathfrak{C}$ is linearly independent. Thus, clearly, it has cardinality $\mathfrak{c}$. Furthermore, any nonzero function
$$
\sum_{i=1}^l \lambda_i \cdot F_i \circ f = \left(\sum_{i=1}^l \lambda_i F_i \right) \circ f
$$
of $\mbox{span}(\mathfrak{C})$ is continuous and surjective, since it is the composition of continuous surjective functions (recall that, from Lemma \ref{lemma2.4}, $\sum_{i=1}^l \lambda_i F_i$ is a continuous surjective function). Therefore, $\mbox{span}(\mathfrak{C})$ only contains, except the zero function, continuous surjective functions.
\end{proof}

\begin{remark}
As we mentioned in the Introduction, and since $\dim \mathcal{C}\left(\mathbb{R}^{m},\mathbb{R}^{n}\right)  = \mathfrak{c}$, this result is the best possible in terms of dimension. The next step (in sense of trying a similar result in higher dimensions) could be related to the lineability of $\mathcal{S}_{m,\mathbb{N}}$ (the set of the continuous surjections from $\mathbb{R}$ onto $\mathbb{R}^\mathbb{N}$ with the product topology). However this is not possible, since $\mathcal{S}_{m,\mathbb{N}} = \emptyset$ (\cite[p. 275]{munkres}).

\end{remark}

\end{document}